\documentclass[11pt,reqno]{amsart}
\usepackage[margin=3 cm]{geometry}
\usepackage[T1]{fontenc}    
\usepackage{hyperref}       
\usepackage[numbers]{natbib}
\usepackage{xcolor}         
\usepackage{bbm}            
\usepackage{amsmath, amssymb, amsfonts, amssymb, amsthm, mathtools, mathrsfs}
\usepackage[shortlabels]{enumitem}

\usepackage{soul}

\newcommand{\Z}{\mathbb{Z}}

\newcommand{\R}{\mathbb{R}}

\newcommand\restr[2]{\left.\kern-\nulldelimiterspace #1 \vphantom{|} \right|_{#2}}
\newcommand{\E}{{\mathbb E}}

\newcommand*\diff{\mathop{}\!\mathrm{d}}

\newcommand{\1}{\mathbbm{1}}
\newcommand{\alphabet}{A}
\newcommand{\subshift}{X}

\renewcommand{\epsilon}{\varepsilon}
\renewcommand{\P}{\mathbb{P}}

\numberwithin{equation}{section}

\theoremstyle{plain}

\newtheorem{thm}{Theorem}[section]
\newtheorem{lem}[thm]{Lemma}
\newtheorem{cor}[thm]{Corollary}
\newtheorem{prop}[thm]{Proposition}

\newtheorem*{thm*}{Theorem}
\newtheorem*{cor*}{Corollary}
\newtheorem*{prop*}{Proposition}
\newtheorem*{lem*}{Lemma}

\theoremstyle{remark}
\newtheorem{remark}[thm]{Remark} 
\newtheorem{defin}[thm]{Definition}

\newtheorem*{defin*}{Definition}

\newtheorem*{remark*}{Remark}
\newtheorem*{eg*}{Example}

\hypersetup{
colorlinks,
linkcolor={red!0!black},
citecolor={green},
urlcolor={blue!0!black}
}

\title{Central Limit Theorems for Finitary Factors of iid Processes}

\author{Raimundo Briceño}
\address{Facultad de Matem\'aticas, Pontificia Universidad Cat\'olica de Chile. Santiago, Chile}
\email{raimundo.briceno@uc.cl}
\thanks{Raimundo Briceño was partially supported by ANID/FONDECYT
Regular 1240508 and Avanza UC AV25085.}
\author{Zemer Kosloff}
\address{School of Mathematics, University of Bristol, Fry Building,
Woodland Road, Bristol BS8 1UG, United Kingdom and The Einstein Institute of Mathematics, The Hebrew University of Jerusalem,
Edmond J. Safra Campus, Givat Ram, Jerusalem 9190401, Israel}
\email{zemer.kosloff@bristol.ac.uk}
\thanks{Zemer Kosloff was partially supported by the ISF (grant No. 1180/22).}
\author{Nicolò Paviato}
\address{Weizmann Institute of Science, 234 Herzl Street, 7610001 Rehovot, Israel}
\email{nicolo.paviato@gmail.com}

\thanks{Nicolò Paviato was partially supported by the ISF (grant No. 264/22).}

\date{}

\begin{document}

\begin{abstract}
We study statistical limit theorems for time-series generated by multidimensional random fields that are finitary factors of iid random processes. Based on the decay of the tail of the coding radius, we establish the Central Limit Theorem with explicit rates of convergence, as well as the weak convergence to a Brownian sheet, for suitably regular functions. Our results can be applied to a variety of classical models in statistical mechanics, such as the Ising model, the hard-core model, and proper colourings.
\end{abstract}

\maketitle

\section{Introduction}

 The Central Limit Theorem (CLT) states that sums of iid sequences, when properly scaled, converge to a normal law. A main topic of study is to establish sufficient conditions on a stationary sequence, possibly dependent, under which the CLT holds. 

A key example of such results arises in smooth dynamics, where one considers a chaotic diffeomorphism $T$ of a Riemmanian manifold $M$ which preserves a probability measure $\mu$. In that setting, it is often the case that for every Hölder continuous function $v\colon M\to \R$, the stationary process $\{v \circ T^n\}_{n \ge 0}$ satisfies the CLT, meaning that there exists $\sigma^2>0$ such that for every $t\in \R$,
 \[
 \lim_{n\to\infty}\mu\left(\frac{\sum_{k=0}^{n-1}v\circ T^k-n\int v\,d\mu}{\sqrt{n\sigma^2}}\leq t\right)=\frac{1}{\sqrt{2\pi}}\int_{-\infty}^te^{-x^2/2}\,dx.
 \]
 
 For a non-conclusive list of results in this direction, as well as some works regarding the functional Central Limit Theorem, see~\cite{Che98,Den89,LimMatMel26,Liv96,Rat73,You98}.

 This work is concerned with obtaining statistical limit theorems for multi-parameter symbolic systems such as Markov random fields and Gibbs measures. The study of the CLT for such systems  has a long history and a wealth of results have been obtained; some relevant ones for this work include~\cite{Bol82,Ded98,IagSou79,New80,Woo83}.
 The techniques in these works are different than the smooth dynamics ones as, instead of martingale techniques, they mostly involve the Stein method and correlation inequalities, such as the FKG inequality. A special model of interest due to its relevance to statistics is functionals of iid random fields, where one can obtain CLT results by approximations with $m$-dependent sequences 
 \cite{Bol82,Ded98,ElMVolWu13,WanWood13}.

In the case of a single transformation, Denker and Keane~\cite{DenKea79} studied a strengthening of the factor map notion called finitary factors. These  are almost everywhere continuous maps between the spaces of the underlying dynamical systems that intertwine the actions. They showed that, if a dynamical system is a finitary factor of an iid system (Bernoulli shift), and the tail of the so called coding radius decays sufficiently fast, then the Central Limit Theorem holds for sufficiently regular functions. In this work, we carry out such a program for multi-parameter symbolic actions which are finitary factors of iid random fields with good coding radius. We show:
\begin{itemize}
\item A CLT result for regular observables, see Theorem~\ref{thm:CLT}.
\item A functional CLT version where the convergence is to a Brownian sheet, see Theorem~\ref{thm:FCLT}.
\item Based on the regularity of the function, rates in the CLT, see Theorem~\ref{thm:rates_CLT}. 
\end{itemize}
Our strategy goes by lifting an observable to the iid random field, and then applying results of~\cite{ElMVolWu13} and~\cite{Gir18}. 

The notion of finitary factors of iid turned out to be of particular relevance for multidimensional symbolic systems arising in statistical mechanics. It was shown by van der Berg and Steif~\cite{vandenberg1999}
that it can serve as a tool for detection of phase transition in lattice models without hard constraints. Spinka and his collaborators~\cite{harel2022,spi20,spi20_SPM} further showed that many lattice models are finitary factors of iid with a coding radius whose tail decays exponentially fast. Applying this to our main theorems, we derive in Subsection~\ref{subsec:stat_mec} the CLT, its functional version, and the speed of convergence for a large class of models from statistical mechanics. 

Our setup and results are presented in Section~\ref{sec:results} and Section~\ref{sec:proofs} is concerned with the proofs.

While working on this manuscript we learned about the preprint~\cite{ChaGalTak25},  where  concentration inequalities for local functions are obtained by finitary coding.

\section{Results}\label{sec:results}

For a finite alphabet $\alphabet$ and positive integer $d$, we consider the configuration space $\alphabet^{\Z^d}$ equipped with the prodiscrete topology. We let $\mathcal{T} = \{T^i\}_{i\in\Z^d}$ be the group of shift operators on $\alphabet^{\Z^d}$, defined as $(T^i x)_j=x_{i+j}$ for $i,j \in \Z^d$. We consider $\subshift\subset\alphabet^{\Z^d}$ closed and shift-invariant (a \textit{subshift}), which is compact. For $k\ge0$, we let $B_k=[-k,k]^d\cap \Z^d$ be the box of radius $k$. The topology on $\subshift$ is generated by the metric $\rho(x,x')=2^{-s(x,x')}$, where 
$$s(x,x')=\inf\{k\ge0 : \restr{x}{B_k} \neq \restr{x'}{B_k}\},$$
with the convention $\inf\emptyset=+\infty$ and $2^{-\infty}=0$.
Writing $\mathcal B$ for the Borel $\sigma$-algebra on $\subshift$, we
fix~$\mu$ a probability measure on $(\subshift,\mathcal{B})$ that is invariant under~$\mathcal{T}$.

The following standard notation will be used throughout the text. For two non-negative sequences $\{a_k\}_{k\ge 1}$ and $\{b_k\}_{k\ge 1}$, we write $a_k\lesssim b_k$, if there exist  $C>0$ and $k_0\ge1$ such that $a_k\le C b_k$ for all $k\ge k_0$. 
If $i\in\Z^d$, we write $\|i\|_\infty=\max\{|i_1|,\dots,|i_d|\}$.

\subsection{Setup}

\subsubsection{Regular Functions}

Given a subset $E\subseteq\Z^d$ and a configuration $x\in\subshift$, we let $\restr{x}{E}\in\alphabet^{E}$ be the restriction of $x$ to $E$.
For a bounded function $v\colon \subshift\to \R$ and $E\subseteq\Z^d$, define
\begin{equation*}\label{eq:delta_k}
\delta_E(v)=\sup\big\{|v(x)-v(x')|: \restr{x}{E}=\restr{x'}{E}\big\}.
\end{equation*}
For $k\geq 0$, we abbreviate $\delta_k(v) = \delta_{B_k}(v)$.
Observe that $\{\delta_k(v)\}_{k\ge0}$ is a non-increasing sequence. The condition $\delta_k(v) \to 0$ is equivalent to uniform continuity of $v$. By compactness of~$\subshift$, this is equivalent to $v$ belonging to the space of real-valued continuous functions, denoted by $C(X)$. In this work we are interested in the following quantified version of continuity.

\begin{defin}\label{def:p_regular}
Let $p\ge0$. A function $v\in C(X)$ is \textit{$p$-regular} if $\sum_{k=1}^\infty k^{p}\delta_k(v)<\infty$.
\end{defin}

If $v$ is H\"older continuous with respect to the metric $\rho$, then~$\delta_k(v)$ decays  to zero exponentially and $p$-regularity is satisfied for any $p\ge0$. 

\begin{remark}\label{rem:sum_volume}
A $(d-1)$-regular function is often referred simply as a regular function \cite{Kel98}. Since  $|B_k\setminus B_{k-1}|$ is asymptotically equivalent to  $k^{d-1}$, it follows that there exists a constant $C_d > 0$ such that for any non-negative sequence $\{a_k\}_{k\ge1}$,
\[
\sum_{i\in\Z^d\setminus{\{\bf0\}}}a_{\|i\|_\infty}=\sum_{k=1}^\infty|B_k\setminus B_{k-1}|a_k
\le C_d \sum_{k=1}^\infty k^{d-1}a_k.
\]
This shows that a function $v$ is $(d-1)$-regular if and only if $\sum_{i\in\Z^d\setminus{\{\bf0\}}}\delta_{\|i\|_\infty}(v)<\infty$.
\end{remark}

\subsubsection{Finitary Coding}\label{subsec:ffiid}

An \textit{iid system}, also known as  a \textit{Bernoulli scheme}, is the shift action on an iid random field. More formally, an iid system is a tuple $(\Lambda^{\Z^d},\mathcal{F}^{\otimes\Z^d},\P^{\Z^d}, \mathcal{S} )$, where $(\Lambda, \mathcal{F}, \P)$ is a standard Borel probability space, $(\Lambda^{\Z^d},\mathcal{F}^{\otimes\Z^d},\P^{\Z^d})$ is the associated product space, and $\mathcal{S} = \{S^i\}_{i \in \Z^d}$ denotes the group of shift operators on $\Lambda^{\Z^d}$.

We say that $(\subshift, \mathcal{B},\mu, \mathcal{T})$ is a \textit{finitary factor of an iid system} (ffiid), if there exists an iid system, a shift-invariant $G_\delta$ set $Y \subseteq \Lambda^{\Z^d}$ with $\P^{\Z^d}(Y)=1$ such that, denoting by $\mathcal{A}= \mathcal{F}^{\otimes\Z^d}\vert_Y$ the induced $\sigma$-algebra on~$Y$, the following two conditions hold:

\begin{enumerate}[I]
\item \textbf{Factor Map.} There exists an $\mathcal A$-measurable map $\pi\colon Y \to \subshift$ such that $\pi_*\big(\restr{\smash{\P^{\Z^d}}}{Y}\big)=\mu$ and $T^i\circ \pi =\pi\circ S^i$ for all $i \in \Z^d$. Such a map is commonly called \emph{factor map}.

\item\label{item:coding_radius}\textbf{Coding Radius.} There exists an $\mathcal A$-measurable function $R\colon Y\to \Z_{\ge0}$, such that whenever  $y,y'\in Y$ satisfy $\restr{y}{B_{R(y)}}=\restr{y'}{B_{R(y)}}$, then $(\pi y)_{\mathbf{0}}=(\pi y')_{\mathbf{0}}$. Such a function is usually called a \textit{coding radius of~$\pi$}.
\end{enumerate}

Henceforth, we assume that the system $(\subshift, \mathcal{B}, \mu,\mathcal{T})$ is an ffiid with factor map $\pi$ and coding radius $R$. 
For the remainder of the paper, we omit the explicit restriction to $Y$ when writing~$\P^{\Z^d}$, and whenever the expectation operator 
$\E$
appears without a subscript, it is implicitly taken with respect to $\P^{\Z^d}$.
We denote the tail of the coding radius $R$ by $\omega(t)=\P^{\Z^d}(R>t)$, $t\in\R$. 
Moreover, we let $R_i=R\circ S^i\colon Y\to\Z_{\ge0}$ be the coding radius at position $i\in\Z^d$, which is well-defined because~$Y$ is $\mathcal{S}$-invariant. Since $S^i$ is measure-preserving, $R_i$ and $R$ are equal in distribution. We remark that, if the configurations $y,y'\in Y$ coincide in the box $i+B_{R_i(y)}$, then $(\pi y)_i=(\pi y')_i$: this follows by the definition of $R$ and the identity
$(\pi S^i y)_{\mathbf{0}}= (T^i\pi y)_{\mathbf{0}}=
(\pi y)_i$.

For integers $0\le m\le k$, we write
\[\textstyle
G_{k,m}=\bigl\{y\in Y:\max_{i\in B_m}R_i(y)\le k-m\bigr\}.
\]
Observe that $G_{k,m}$ is a subset of configurations $y$ for which $\restr{\pi(y)}{B_m}$ is determined by the values~$\restr{y}{B_k}$. By stationarity,
\begin{equation}\label{eq:non_good_set}
\P^{\Z^d}(G^c_{k,m})\le \sum_{i\in B_m}\P^{\Z^d}(R_i>k-m)=
|B_m|\ \P^{\Z^d}(R>k-m).
\end{equation}

\subsection{Main Results}\label{sec:main_results}

Let $\{\Gamma_n\}_{n\ge0}$ be a sequence of finite subsets of $\Z^d$. We assume that the sequence is \textit{Følner}, that is, 
\[
\lim_{n\to\infty} \frac{|\Gamma_n\cap (\Gamma_n-j)|}{|\Gamma_n|} = 1, \qquad \text{for each } j \in \Z^d.
\]
We denote by $\longrightarrow_d$ convergence in distribution, and  by $\mathcal N(0,\sigma^2)$ a Gaussian random variable with mean zero and variance~$\sigma^2>0$.

\begin{thm}[CLT]\label{thm:CLT}
Let $v\in C(X)$ with $\E_\mu[v]=0$, and assume that there exist $p\ge d-1$ and $s>2d+d^2/(p+1)$ such that $v$ is $p$-regular and $\E[R^s]<\infty$. 
Then, the series 
\begin{equation}\label{eq:sigma^2}
\sigma_v^2 = \sum_{i\in\Z^d}\E_\mu[v(v\circ T^i)]
\end{equation}
converges and is non-negative. Moreover, if $\sigma_v^2>0$, then
\[
|\Gamma_n|^{-1/2}\sum_{i\in\Gamma_n}v\circ T^i\longrightarrow_d\mathcal N(0,\sigma_v^2).
\]
\end{thm}

In our next results, under more stringent conditions on the regularity of $v$ and the tail of the coding radius, we will prove a functional CLT and describe rates of convergence in the CLT. We make a note that under those conditions the series from~\eqref{eq:sigma^2} converges and is non-negative. 

For $t\in[0,1]^d$, we define $A_t=\prod_{i=1}^d[0,t_i]$.
For a fixed $v\colon \subshift\to\R$ and any $n\ge0$, we define the process $W_n = \{W_{n,t}(v)\}_{t\in[0,1]^d}$ as
\begin{equation*}
W_{n,t}(v)=\sum_{i\in\{1,\dots,n\}^d}\lambda(nA_t\cap Q_i)(v\circ T^i),    
\end{equation*}
where $Q_i=(i_1-1,i_1]\times\dots\times(i_d-1,i_d]$ and $\lambda$ is the Lebesgue measure on $\R^d$. 

A standard \textit{Brownian sheet} on $[0,1]^d$  is a zero-mean Gaussian random field $W=\{W_t\}_{t\in[0,1]^d}$ with covariance 
$\E[W_sW_t]=\prod_{i=1}^d\min(s_i,t_i)$,  see~\cite{Kho02}. We view $W_n$ and $W$ as random elements on the space of continuous functions $\mathcal C([0,1]^d)$, endowed with the supremum norm. We denote by $\longrightarrow_w$ weak convergence for probability measures on this space.

\begin{thm}[Convergence to Brownian Sheet]\label{thm:FCLT} Let $v \in C(X)$ with $\E_\mu[v]=0$. Suppose that there exist $p\ge d-1$ and $s>2d^2+d^2/(p+1)$ such that $v$ is $p$-regular and $\E[R^s]<\infty$. 
If $\sigma_v^2$ from \eqref{eq:sigma^2} is positive, then $n^{-d/2}W_n(v)\longrightarrow_w\sigma_v W$. 
\end{thm}

\begin{remark}\label{remark:CLT_and_FCLT}
Note that in both Theorems~\ref{thm:CLT} and~\ref{thm:FCLT} one can trade regularity of the observable
for a weaker moment condition on $R$. For a zero-mean function $v\colon \subshift\to\R$, we obtain the CLT and convergence to  Brownian sheet in the following two extreme cases:
\begin{enumerate}[(i)]
\item\label{item:minimal} $v$ is $(d-1)$-regular, and $\E[R^s]<\infty$ for some $s>3d$ for the CLT, or $s>2d^2+d$ for convergence to Brownian sheet.
\item\label{item:maximal} $v$ is $p$-regular for every $p\ge0$, and $\E[R^s]<\infty$ for some $s>2d$ for the CLT, or $s>2d^2$ for convergence to Brownian sheet.
\end{enumerate}
Point~\ref{item:minimal} follows simply by taking $p=d-1$ in the theorems, whereas point~\ref{item:maximal} follows by choosing~$p$ sufficiently large.
\end{remark}

We let $\Phi \colon \mathbb{R} \to [0,1]$ denote the cumulative distribution function of a standard Gaussian.  To state effective rates in the CLT, we assume that $\Gamma_n=[0,n-1]^d\cap\Z^d$.
For $\epsilon>0$, we define the function $h_\epsilon\colon \R^+\to\R$ as
\begin{equation}\label{eq:h_eps(q)}
h_\epsilon(q) = \bigg(\frac32-\frac1q\bigg)d+\bigg(\frac{d^2}{\epsilon}-2d\bigg)\bigg(\frac{q+1}{q}\bigg)-\frac12.
\end{equation}

\begin{thm}[Rates in the CLT]\label{thm:rates_CLT}
Let $v\in C(X)$ with $\E_\mu[v]=0$ and  let $\epsilon\in(0,d/2)$. Suppose there exists $q \ge 3$ such that $v$ is $p$-regular and $\E[R^s] < \infty$ for some parameters $p$ and $s$ satisfying
\begin{equation}\label{eq:conditions_p_and_s}
p  \ge h_\epsilon(q) 
\qquad \text{and} \qquad 
s > \big(h_\epsilon(q) + 1\big)\big(q + d/(p+1)\big).
\end{equation}
If $\sigma_v^2$ from \eqref{eq:sigma^2} is positive, then there is $C>0$ such that for any $n\ge1$,
\[
\sup_{t\in\R}\bigg|\mu\bigg(\frac{\sum_{i\in \Gamma_n} v\circ T^i}{n^{d/2}}\le t\bigg)-\Phi(t/\sigma_v)\bigg|\le C\max\{n^{-d/2+\epsilon },n^{-1}\}.
\]
\end{thm}

For any $\epsilon\in(0,d/2)$, we define
\[\textstyle
p_{\min}(\epsilon) = \inf_{q \ge 3} h_\epsilon(q) =
\begin{cases}
\frac{d^2}{\epsilon} - \frac{d+1}{2} & \text{if } \epsilon \in (0, d/3] \\[1ex]
\frac{4d^2}{3\epsilon} - \frac{3d+1}{2} & \text{if } \epsilon \in (d/3, d/2).
\end{cases}
\]

\begin{cor}
\label{cor:CLT-rate}
Let $v\in C(X)$ with $\E_\mu[v]=0$. Assume that $\E[R^s]<\infty$ for all $s\ge1$. For any $\epsilon\in(0,d/2)$,   
suppose that $v$ is $p$-regular with $p\ge p_{\min}(\epsilon)$. If $\sigma_v^2$ from \eqref{eq:sigma^2} is positive, there is $C>0$ such that for any $n\ge1$,
\[
\sup_{t\in\R}\bigg|\mu\bigg(\frac{\sum_{i\in \Gamma_n} v\circ T^i}{n^{d/2}}\le t\bigg)-\Phi(t/\sigma_v)\bigg|\le C\max\{n^{-d/2+\epsilon },n^{-1}\}.
\]
\end{cor}

\begin{proof}
Since the parameter $s$ can be chosen arbitrarily large, the moment condition for $R$ in Theorem~\ref{thm:rates_CLT} is trivially satisfied. So, it suffices to focus on the lower bound for $p$, which is precisely~$p_{\min}(\epsilon)$. The result then follows immediately.
\end{proof}

\begin{remark}
In comparison with the classical Berry-Esseen rates of order $O(n^{-d/2})$ (see for example~\cite[Chap.\,3, Sec.\,11]{Shi16}), 
the null sequence $\max\{n^{-d/2+\epsilon}, n^{-1}\}$ introduces a dimensional constraint. The term of order $O(n^{-1})$ 
dominates if and only if $\epsilon \le (d-2)/2$.
In dimensions $d=1,2$ this is impossible and the rate of convergence is $O(n^{-d/2+\epsilon})$.
For $d \ge 3$, the threshold $\epsilon = (d-2)/2 > 0$ is achievable. Under the assumptions of Corollary~\ref{cor:exp-ssm}, evaluating~$p_{\min}$ at this critical value gives the required regularity of $v$ to hit the $O(n^{-1})$ barrier:
\[
p_{\min}\bigg(\frac{d-2}{2}\bigg) =
\begin{cases}
\frac{2d^2}{d-2} - \frac{d+1}{2} & \text{if } 3 \le d \le 6, \\
\frac{8d^2}{3(d-2)} - \frac{3d+1}{2} & \text{if } d \ge 7.
\end{cases}
\]
For instance, for $d=3$, any observable with $p > 16$ attains the $O(n^{-1})$ rate, which is the best that we can achieve for our method.
\end{remark}

\subsection{Applications to Statistical Mechanics}\label{subsec:stat_mec}

Let $\mu$ be a shift-invariant Markov random field on $\Z^d$ with topological support $X \subseteq A^{\Z^d}$. Given finite sets $V,W \subseteq \Z^d$, let $\operatorname{dist}(V,W)$ be the graph distance from $V$ to $W$ in $\Z^d$ according to the canonical Cayley graph, and let $\partial V$ denote all the points at distance $1$ from $V$. Let $X_{\partial V} = \{x|_{\partial V} : x \in X\} \subseteq A^{\partial V}$. For $\tau \in X_{\partial V}$, define the finite-volume conditional distribution
\[
  P_V^\tau(\eta) := \mu\bigl( \{x\in X:x|_V=\eta\}
    \,\big|\,
    \{x\in X:x|_{\partial V}=\tau\}
  \bigr),
  \qquad \eta\in A^V.
\]
Given $U\subseteq V$, let $P_{V,U}^\tau$ denote its marginal on $A^U$. If $\tau,\tau'\in X_{\partial V}$, set
\[
\Sigma_V(\tau,\tau')
:=
\{i\in\partial V:\tau_i\ne\tau'_i\}.
\]
We say that $\mu$ satisfies \emph{exponential strong spatial mixing} if there are constants $C,c>0$ such that
\begin{equation}
\label{eq:exp-ssm}
\bigl\|P_{V,U}^{\tau}-P_{V,U}^{\tau'}\bigr\|_{\rm TV} \le C|U|\exp\left(
-c\,\operatorname{dist} \bigl(U,\Sigma_V(\tau,\tau')\bigr)
\right)
\end{equation}
for every finite $V\subseteq\Z^d$, every $U\subseteq V$, and every
$\tau,\tau'\in X_{\partial V}$, with the convention that $\operatorname{dist}(U,\varnothing)=+\infty$.

In \cite[Theorem 1.1]{spi20}, Spinka proved that, for $d \geq 2$, if $\mu$ satisfies exponential strong spatial mixing, then $\mu$ is an ffiid where the tail of the coding radius decays exponentially. In particular, this implies that $R$ has moments of all orders.

\begin{cor}
\label{cor:exp-ssm}
Suppose that $\mu$ satisfies exponential strong spatial mixing, and let $v \in C(X)$ be $(d-1)$-regular with $\E_\mu[v]=0$. Then $\sigma_v^2$ from \eqref{eq:sigma^2} converges and is non-negative. Moreover, if $\sigma_v^2 > 0$, then
\[
|\Gamma_n|^{-1/2}\sum_{i\in\Gamma_n}v\circ T^i\longrightarrow_d\mathcal N(0,\sigma_v^2)
\]
for any Følner sequence $\{\Gamma_n\}_{n \ge 1}$ for $\Z^d$, and
\[
n^{-d/2}W_n(v)\longrightarrow_w\sigma_v W.
\]
\end{cor}

\begin{proof}
By \cite[Theorem 1.1]{spi20}, $\mu$ is an ffiid with $\mathbb E[R^s] < \infty$ for every $s \ge 0$. In view of Remark~\ref{remark:CLT_and_FCLT}\ref{item:minimal}, both Theorem \ref{thm:CLT} and Theorem \ref{thm:FCLT} apply for $(d-1)$-regular functions on $X$.
\end{proof}

\begin{cor}
\label{cor:exp-ssm-2}
Suppose that $\mu$ satisfies exponential strong spatial mixing, and let $v \in C(X)$ be Hölder continuous with $\E_\mu[v]=0$. Then $\sigma_v^2$ from \eqref{eq:sigma^2} converges and is non-negative. Moreover, if $\sigma_v^2 > 0$, then for any $\epsilon\in(0,d/2)$ there is $C>0$ such that for any $n\ge1$ and $\Gamma_n = [0,n-1]^d \cap \Z^d$,
\[
\sup_{t\in\R}\bigg|\mu\bigg(\frac{\sum_{i\in \Gamma_n} v\circ T^i}{n^{d/2}}\le t\bigg)-\Phi(t/\sigma_v)\bigg|\le C\max\{n^{-d/2+\epsilon },n^{-1}\}.
\]
\end{cor}

\begin{proof}
By \cite[Theorem 1.1]{spi20}, $\mu$ is an ffiid with $\mathbb E[R^s] < \infty$ for every $s \ge 0$. As $v$ is Hölder continuous, it is $p$-regular for every $p \ge 0$. It remains to apply Corollary~\ref{cor:CLT-rate}.
\end{proof}

We next list some models and regimes that satisfy exponential strong spatial mixing, and therefore Corollaries~\ref{cor:exp-ssm}~and~\ref{cor:exp-ssm-2} apply.

\begin{enumerate}
\item \emph{Ising model.} The unique Gibbs measure of the ferromagnetic Ising model at inverse temperature $0\le \beta< \beta_c(d)$ satisfies exponential strong spatial mixing. See \cite[Corollary~1.8]{ding2023}.

\item \emph{Proper colourings.} Let $\alpha>1$ be the solution of $\alpha^{\alpha}=e$, and set $\gamma := \frac{4\alpha^3-6\alpha^2-3\alpha+4}{2(\alpha^2-1)}$. The unique Gibbs measure for proper $k$-colourings of $\Z^d$ satisfies exponential strong spatial mixing whenever $k>2\alpha d-\gamma$. See \cite[Section~1.1.1]{spi20_SPM} and the references therein.

\item \emph{Hard-core model.} Define $\lambda_c(\Delta) := \frac{\Delta^\Delta}{(\Delta-1)^{\Delta+1}}$, and let $\mu_d$ denote the connective constant of $\Z^d$. The unique hard-core Gibbs measure with activity $\lambda$ satisfies exponential strong spatial mixing whenever $\lambda<\lambda_c(\mu_d)$. See \cite{sinclair2017}.

\item \emph{Graph-homomorphism spaces.} Let $H$ be a finite constraint graph. In \cite[Proposition~5.3]{briceno2017} it was proven that if $\operatorname{Hom}(\Z^d,H)$ has the \emph{unique maximal configuration property}, then the associated Gibbs specification constructed therein satisfies exponential strong spatial mixing at sufficiently large activity. 

\item \emph{Further finite-state models.} The same argument applies to the $q$-type Widom--Rowlinson model and to the beach model. See also \cite{regts2023} for general sufficient conditions.
\end{enumerate}

One can apply the main Theorems~\ref{thm:CLT}, \ref{thm:FCLT}, and \ref{thm:rates_CLT} to the following models, for which the required finitary codings are obtained by separate arguments.

\begin{enumerate}
\item \emph{Random-cluster and Potts models.} Let $q\ge1$. For $0\le p<p_c(q,d)$, the random--cluster measure on $\Z^d$ admits a finitary coding from an iid process whose coding radius has exponential tails. Likewise, for integer $q\ge2$, the unique Gibbs measure of the ferromagnetic $q$-state Potts model at inverse temperature $0\le \beta< \beta_c(q,d)$ admits such a coding. See \cite{harel2022}.

\item \emph{Long-range Ising model.} It was established in \cite[Appendix~A]{Corentin25} that, for $\alpha>2d$ and small enough inverse temperature $\beta$, the long-range ferromagnetic Ising model with couplings constants satisfying $\left|J_{i j}\right| = O\left(\|j-i\|_1^{-\alpha}\right)$ for all $i, j \in \Z^d$, is a ffiid where the coding radius decays at least as fast as $t^{1-\alpha/d}$. Thus, in this regime, the model has polynomial tails and the coding radius has finite $s$-moment for $s < \alpha/d - 2$.

\item \emph{Finite-state edge models.} The results also apply to edge models such as the monomer--dimer model. See \cite{vandenberg1999}.

\end{enumerate}

\section{Proofs}\label{sec:proofs}

In the current section, we assume that the symbolic system $(X,\mathcal B,\mu,\mathcal T)$ is an ffiid, following the conventions introduces in Section~\ref{subsec:ffiid}.

\subsection{Dependence Measures on Random Fields}

Consider the coupled probability space $(\Lambda^{\Z^d}\times \Lambda,\P^{\Z^d}\otimes \P)$ and the function $\phi\colon \Lambda^{\Z^d}\times \Lambda\to \Lambda^{\Z^d}$,
\[
(\phi(x,a))_i=\begin{cases}
x_i&\text{if } i\neq \mathbf 0\\
a&\text{if } i=\mathbf{0}
\end{cases}
.\]
 The following definition, adapted from~\cite{Wu05}, are used to measure the change of a random field when the zero coordinate of a configuration is replaced by an independent copy.
\begin{defin}
\label{def:dependence_class}
For $q\ge1$ and $g\in L^q(\Lambda^{\Z^d})$, we define the \textit{physical dependence measure} of $g$ at site $i\in\Z^d$ as 
\begin{align*}
\theta_{i,q}(g)&=\bigg(\int_{\Lambda^{\Z^d}\times \Lambda}\big|g(S^ix)-g(S^i(\phi(x,a)))\big|^q\diff \big(\P^{\Z^d}\otimes \P\big)(x,a)\bigg)^{1/q}
\\
&=\big\|g\circ S^i-g\circ S^i\circ\phi
\big\|_{L^q(\Lambda^{\Z^d}\times \Lambda)}.   
\end{align*}

We say that the random field $\{g\circ S^i\}_{i\in\Z^d}$ is \textit{of dependence class $q$} if
\[
\Delta_q=\sum_{i\in\Z^d}\theta_{i,q}(g)<\infty.
\]
\end{defin}

\noindent 
While both~\cite{ElMVolWu13} and~\cite{Wu05} refer to the finiteness of $\Delta_q$ as \textit{q-stability}, we adopt the term dependence class 
$q$ to avoid any potential confusion with the distinct theory of $\alpha$-stable processes.
As we will discuss later, a field of dependence class $2$ satisfies the Central Limit Theorem~\cite[Theorem 1]{ElMVolWu13}, whereas the sufficient conditions of as functional CLT in~\cite[Theorem 2]{ElMVolWu13} require a higher dependence class. 

The next result is the first step in connecting the regularity of a function $v\in C(X)$ with the dependency class of the field generated by its lift.

\begin{prop}\label{prop:technicality}
Let $v \in C(X)$, $i\in\Z^d\setminus\{\bf0\}$, and  $m\in\{0,\dots,\|i\|_\infty-1\}$. Assuming that $y\in S^{-i}G_{\|i\|_\infty-1,m}$, then for every configuration $y'\in Y$ such that $y_j=y'_j$ for all $j\in\Z^d\setminus\{\mathbf 0\}$, 
\[
|v\circ\pi(S^{i}y)-v\circ\pi(S^{i}y')|\le \delta_m(v).
\]
\end{prop}

\begin{proof} For $0\le n\le k$ and $y\in Y$, we write $[y]_k=\{y'\in Y:\restr{y}{B_k}=\restr{y'}{B_k}\}$
for the cylinder set of~$y$ on the shape $B_k$.
If $y\in G_{k,n}$, then for any $y'\in [y]_k$ the configurations~$\pi(y)$ and  $\pi(y')$ coincide on the whole~$B_{n}$, and we have the useful estimate
\begin{equation*}\label{eq:delta_m}
|v\circ\pi(y)-v\circ\pi(y')|\le\delta_n(v).
\end{equation*}

Assuming now $i$, $m$, $y$, and $y'$ as in the statement, we note that $S^iy'\in[S^iy]_{\|i\|_\infty-1}$. The proof is finished by the estimate above.
\end{proof}

Fix a function $v\in C(X)$. For $q\in[1,\infty)$ and a positive integer $k$, we define
\begin{equation}\label{eq:rho_k}\textstyle
\xi_k^{(q)}(v)=\min_{0\le m\le k-1}\big(\delta_m(v)+|B_m|^{1/q}\omega^{1/q}(k-1-m)\big).
\end{equation}

\begin{prop}\label{prop:regularity_of_rho}
Let $r\ge0$ and $q\ge1$. Assume that  $p\ge{r}$ and $s>(r+1)(q+d/(p+1))$ such that $v$ is $p$-regular and $\E[R^s]<\infty$. Then, $\sum_{k=1}^\infty k^{r}\xi_k^{(q)}(v)<\infty$.
\end{prop}

\begin{proof}
Since $|B_m|=(2m+1)^d$, the conclusion  is equivalent to the existence of a sequence of integers $0\le m_k\le k-1$ such that
\begin{equation*}
\sum_{k=1}^\infty k^{{r}}\delta_{m_k}(v)<\infty
\qquad\text{and}\qquad
\sum_{k=1}^\infty k^{{r}}m_k^{d/q}\omega^{1/q}(k-m_k)<\infty.
\end{equation*}
We now show that this condition is satisfied 
with
$m_k=\lfloor k^{\alpha}/2\rfloor$, $\alpha=(r+1)/(p+1)$. For $n\ge1$, we define $$M(n)=\{k\ge1:m_k=n\}.$$ 
Since $k\in M(n)$ if and only if $2n\le k^{\alpha}<2n+2$, the cardinality of $M(n)$ is bounded by 
$(2n+2)^{1/\alpha}-(2n)^{1/\alpha}+1$. 
By applying the mean value theorem to the function $x \mapsto x^{1/\alpha}$ on the interval $[2n, 2n+2]$, we get
\[
|M(n)|\le (2n+2)^{1/\alpha}-(2n)^{1/\alpha}+1 \lesssim n^{1/\alpha-1}.
\]
Note that $k\le (4n)^{1/\alpha}$ for any $n\ge1$ and $k\in M(n)$.
Since $(1+r)/\alpha-1=p$, 
\[\begin{split}
\sum_{k=1}^\infty k^{{r}}\delta_{m_k}(v)&=\sum_{n=1}^\infty
\sum_{k\in M(n)} k^{{r}}\delta_n(v) \lesssim
\sum_{n=1}^\infty\delta_n(v)|M(n)|
 (n^{1/\alpha})^{{r}}\\
&\lesssim 
\sum_{n=1}^\infty n^{1/\alpha-1}n^{r/\alpha} \delta_n(v)=
\sum_{n=1}^\infty n^{p} \delta_n(v)<\infty.
\end{split}
\]\

Let us define $\epsilon > 0$ by  $q\epsilon = s - ((r+1)q + d\alpha)$.
By the moment condition on $R$, there is~$C>0$ such that, for all large $t$, 
\[
\omega(t)\le Ct^{-q({r}+1+\epsilon)
-d\alpha}
\big.
\]
Since $k-m_k\ge k/2$ and $\omega$ is non-increasing,
\[\begin{split}
\sum_{k=1}^\infty k^{{r}}m_k^{d/q}\omega^{1/q}(k-m_k)&\le
\sum_{k=1}^\infty
k^{{r}}k^{d\alpha/q}\omega^{1/q}(k/2)\\
&\le
C\sum_{k=1}^\infty k^{{r}+d\alpha/q}
k^{-{r}-1-\epsilon-d\alpha/q}\\
&= C\sum_{k=1}^\infty k^{-(1+\epsilon)}<\infty.
\end{split}
\]
The proof is finished.
\end{proof}

In the following results, we let $f=v\circ\pi$ and bound the physical dependence measure of $f$ by means of the sequence $\{\xi_k^{(q)}(v)\}_{k\ge1}$.

\begin{prop}\label{prop:rho_dep_class_q}
There exists $C>0$ such that 
\begin{equation}\label{eq:bound_phys_dep_with_xi}
\theta_{i,q}(f)\le C^{1/q}\, \xi_{\|i\|_\infty}^{(q)}(v),
\end{equation}
for any $q\in[1,\infty)$ and  $i\in\Z^d\setminus{\bf{\{0\}}}$. In particular, if there exists $q\in[1,\infty)$ such that $\sum_{k=1}^\infty k^{d-1}\xi_k^{(q)}(v)<\infty$, then the random field $\{f\circ S^i\}_{i\in\Z^d}$ is of dependence class $q$.
\end{prop}
\begin{proof}
For $i\in\Z^d\setminus\{\bf0\}$ and $m\in\{0,\dots,\|i\|_\infty-1\}$, we define
\[
J_i=\big(S^{-i}(G_{\|i\|_\infty-1,m})\times \Lambda\big)\cap \phi^{-1}Y=\big\{(y,a)\in S^{-i}\big(G_{\|i\|_\infty-1,m}\big)\times \Lambda: \phi(y,a)\in Y\big\}.
\]
Since $Y$ has full measure and $\phi_*\big(\P^{\Z^d}\times\P\big)=\P^{\Z^d}$, it follows that $\P^{\Z^d}\otimes \P(J_i)=\P^{\Z^d}\big(G_{\|i\|_\infty-1,m}\big)$.
For any $(y,a)\in J_i$ we have $y\in S^{-i}\big(G_{\|i\|_\infty-1,m}\big)$, and the configurations $y$ and $\phi(y,a)$ are equal in every coordinate except at the origin. By Proposition~\ref{prop:technicality},
\[
\big|f(S^i(y))-f(S^i(\phi(y,a)))\big|\le \delta_m(v),
\]
and so
\begin{align*}
\E_{\P^{\Z^d}\otimes \P}\big[\big|f\circ S^i-f\circ S^i\circ \phi\big|^q\1_{J_i}\big]\le 
\delta_m^q(v).
\end{align*}
Using~\eqref{eq:non_good_set},
\[
\E_{\P^{\Z^d}\otimes \P}\big[\big|f\circ S^i-f\circ S^i\circ \phi\big|^q\1_{J_i^c}\big]\le
2\|v\|_\infty\,\P^{\Z^d}\big(G^c_{\|i\|_\infty-1,m}\big)\le 2\|v\|_\infty |B_m|\, \omega(\|i\|_\infty-1-m).
\]
Therefore, there is $C>0$ such that
\[
\theta_{i,q}^q(f)\le C\big( \delta_m^q(v)+
|B_m|\, \omega(\|i\|_\infty-1-m)\big).
\]
The statement follows by elevating to the power of $1/q$, using the concavity of the function $x\mapsto x^{1/q}$, $x\ge0$, and taking the minimum over the possible range of $m$. Finally, observe that
\begin{align*}
\Delta_q=\sum_{i\in\Z^d}\theta_{i,q}(f)&\le 
\theta_{{\bf 0},q}(f)+
C^{1/q}\sum_{i\in\Z^d\setminus\{{\bf 0}\}}
\xi_{\|i\|_\infty}^{(q)}(v).
\end{align*}
Reasoning as in Remark~\ref{rem:sum_volume} and since $\theta_{{\bf 0},q}(f)\le 2\|v\|_\infty$, the summability condition on $\xi_k^{(q)}$ yields that  $\Delta_q<\infty$.
\end{proof}

Proposition \ref{prop:rho_dep_class_q} enables us to deduce the dependence class of the random field $\{f\circ S^i\}_{i\in\Z^d}$ from the regularity of the observable $v$ and the moments of the coding radius $R$.
\begin{cor}\label{cor:dep_class_q}
Let $q\ge1$, $p\ge d-1$, and $s>dq+d^2/(p+1)$. If $v$ is $p$-regular and $\E[R^s]<\infty$, then the random field $\{f\circ S^i\}_{i\in\Z^d}$ has dependence class $q$.
\end{cor}

\begin{proof}
By applying Proposition~\ref{prop:regularity_of_rho} with ${r}=d-1$, we get that $
\sum_{k=0}^\infty k^{d-1}\xi_k^{(q)}<\infty$.
Our result follows by Proposition~\ref{prop:rho_dep_class_q}.
\end{proof}

\subsection{Central Limit Theorems}
Recall concepts from previous section.

\begin{lem}\label{lem:CLT_from_stability}
Let $g\in L^2(\Lambda^{\Z^d})$ with $\E[g]=0$, and assume that the random field $\{g\circ S^i\}_{i\in\Z^d}$ is of dependence class $2$. Then, the series
\begin{equation}\label{eq:sigma^2_general}
\sigma_g^2=\sum_{i\in\Z^d}\E[g(g\circ S^i)]
\end{equation}
converges and is non-negative. Moreover, if $\sigma_g^2>0$, then
\[
|\Gamma_n|^{-1/2}\sum_{i\in\Gamma_n}g\circ S^i\longrightarrow_d\mathcal N(0,\sigma_g^2)
\]
for every Følner sequence $\{\Gamma_n\}_{n \ge 1}$.
\end{lem}
\begin{proof}
It follows immediately from~\cite[Proposition 2]{ElMVolWu13}, and~\cite[Theorem 1]{ElMVolWu13} and the paragraph after it.
\end{proof}

\begin{proof}[Proof of Theorem~\ref{thm:CLT}]
Since $\pi$ is a factor map, $\pi_*\big(\P^{\Z^d}\big)=\mu$ and $\pi\circ S^i=T^i\circ\pi$, $\E[f]=0$, and the CLT for $v$ is equivalent to the one for $f=v\circ\pi$. The assumptions of Corollary~\ref{cor:dep_class_q} are satisfied with $q=2$, therefore the field $\{f\circ S^i\}_{i\in\Z^d}$ has dependence class $2$. Lemma~\ref{lem:CLT_from_stability} yields that the series 
\[
\sigma_f^2=\sum_{i\in\Z^d}\E[f(f\circ S^i)]
\]
converges and is non-negative.
Using again that $\pi$ is a factor map, it follows that
\[
\E\big[f(f\circ S^i)\big]=
\E\big[(v(v\circ T^i))\circ\pi\big]=
\E_\mu[v(v\circ T^i)],
\]
which leads to 
\begin{equation}\label{eq:two_sigmas}\sigma_f^2=\sigma_v^2,
\end{equation}
where $\sigma_v^2$ is from~\eqref{eq:sigma^2}.
By the assumption $\sigma_v^2>0$, Lemma~\ref{lem:CLT_from_stability} gives the CLT for~$f$, finishing the proof.
\end{proof}

For a measurable $g\colon \Lambda^{\Z^d}\to\R$ and $n\ge1$, we define the process 
$\widehat W_{n}(g)=\{\widehat W_{n,t}(g)\}_{t\in[0,1]^d}$ as
\begin{equation*}
\widehat W_{n,t}(g)=\sum_{i\in\{1,\dots,n\}^d}\lambda(nA_t\cap Q_i)(g\circ S^i).    
\end{equation*}

\begin{lem}\label{lem:FCLT_from_dep_class}
Let $q>2d$,  $g\in L^q(\Lambda^{\Z^d})$ with $\E[g]=0$, and let $\sigma_g^2$ from~\eqref{eq:sigma^2_general} converge and be positive. If the random field $\{g\circ S^i\}_{i\in\Z^d}$ has dependence class $q$, then $n^{-d/2}\widehat W_n(g)\longrightarrow_w \sigma^g W$.
\end{lem}
\begin{proof}
The process $\widehat W_n(g)$ coincides with the one in~\cite[Equation~(6)]{ElMVolWu13}, using the family of Borel subsets $\mathcal Q_d=\{A_t:t\in[0,1]^d\}$. As remarked in the paragraph before~\cite[Theorem~2]{ElMVolWu13}, the Vapnik–Chervonenkis dimension of $\mathcal Q_d$ is $V=d+1$. Since $\{g\circ S^i\}_{i\in\Z^d}$ has dependence class $q>2(V-1)$, the claim follows by~\cite[Theorem~2(i)]{ElMVolWu13}. 
\end{proof}

\begin{proof}[Proof of Theorem~\ref{thm:FCLT}]
Let $f=v\circ\pi$.
Since $\pi$ is measure-preserving, $\mathbb E[f]=0$, and $W_{n}(v)$ and $W_n(v)\circ\pi$ share the same distribution. Moreover, by $\pi\circ S^i=T^i\circ\pi$ for any $i\in\Z^d$, it follows that $W_n(v)\circ\pi=\widehat W_n(f)$.
As $s > 2d^2 + d^2/(p+1)$, there exists a parameter $q>2d$ such that $s > dq + d^2/(p+1)$.
The assumptions of Corollary~\ref{cor:dep_class_q} are satisfied for such a $q$, and hence the field $\{f\circ S^i\}_{i\in\Z^d}$ has dependence class $q$. Consider $\sigma_f^2$ and $\sigma_v^2$ respectively from \eqref{eq:sigma^2_general} and~\eqref{eq:sigma^2}. Similarly to~\eqref{eq:two_sigmas}, by our assumptions we get $\sigma_f^2=\sigma_v^2>0$. Consequently, Lemma~\ref{lem:FCLT_from_dep_class} yields that $n^{-d/2}\widehat W_n(f)\longrightarrow_w\sigma_v W$.\end{proof}

\begin{remark}
Following~\cite[Theorem~2]{ElMVolWu13}, Theorem~\ref{thm:FCLT} can be extended to a functional Central Limit Theorem indexed by a more general family of Borel subsets of $[0,1]^d$. Specifically, if the indexing family forms a Vapnik-Chervonenkis class with index $V$ (see~\cite{vdVWel23} for a definition), the weak convergence holds provided that the dependence class satisfies $q>2(V-1)$. We chose to state the result only for the Brownian sheet (where $V=d+1$) to keep the exposition simple and focus on the interplay between the regularity of the observable and the tail of the coding radius.
\end{remark}

\subsection{Bounds for Covariances}

In this subsection we obtain bounds on covariances, which will be used later in Proposition~\ref{prop:correlation_decay} in order to obtain rates in the CLT.

\begin{prop}\label{prop:covariance_tail}
Let $g\in L^2(\Lambda^{\Z^d})$ with $\E[g]=0$. Then, for any $n\ge0$,
\[
\sum_{\|j\|_\infty \ge n} \big|\E[g(g\circ S^j)]\big| \le 2 \Delta_2 \sum_{\|j\|_\infty \ge n/2} \theta_{j,2}(g).
\]
\end{prop}
\begin{proof}
Let $\tau\colon\Z\to\Z^d$ be a bijection. For $\ell\in\Z$ and $j\in\Z^d$, we define as in~\cite[Equation~(9)]{ElMVolWu13},
\[
P_\ell(g\circ S^j)=\E[g\circ S^j\mid\mathcal F_\ell]-
\E[g\circ S^j\mid\mathcal F_{\ell-1}],
\]
where $\mathcal F_\ell$ is the $\sigma$-algebra generated by the coordinate maps $x\mapsto x_{\tau(r)}$, for $r\le \ell$. Since $g$ has mean zero, it follows that $g\circ S^j=\sum_{\ell\in\Z} P_{\ell}(g\circ S^j)$, where the series converges in $L^2$. As in the proof of~\cite[Proposition 2]{ElMVolWu13}, we note that $\E\big[P_\ell(g)P_h(g\circ S^j)\big]=0$ if $\ell\neq h$, and so
\[
\E[g(g\circ S^j)]=\sum_{\ell\in\Z} \E\big[P_\ell (g)P_\ell(g\circ S^j)\big].
\]
By~\cite[Lemma 1]{ElMVolWu13}, we have $\|P_\ell(g\circ S^j)\|_2\le\theta_{j-\tau(\ell),2}(g)$. So,
applying Cauchy-Schwarz,
\[
\big|\E[g(g\circ S^j)]\big| 
\le \sum_{\ell\in\Z} \|P_\ell(g)\|_2\|P_\ell(g\circ S^j)\|_2
\le  \sum_{\ell\in\Z} \theta_{-\tau(\ell)   ,2}(g)  \theta_{j-\tau(\ell),2}(g).
\]
Using Tonelli,
\[
\sum_{\|j\|_\infty\ge n} \big|\E[g(g\circ S^j)]\big| \le
\sum_{\ell\in\Z} \theta_{-\tau(\ell),2}(g)
\sum_{\|j\|_\infty \ge n}\theta_{j-\tau(\ell),2}(g).
\]
Since $\tau$ is a bijection, we can change variable in the outer sum as $k = -\tau(\ell) \in \Z^d$. Writing $i = j + k$, the inequality becomes
\[
\sum_{\|j\|_\infty\ge n} \big|\E[g(g\circ S^j)]\big| \le
\sum_{k\in\Z^d} \theta_{k,2}(g) \sum_{i:\|i-k\|_\infty \ge n} \theta_{i,2}(g).
\]
By $\|i-k\|_\infty \le \|i\|_\infty + \|k\|_\infty$, the condition $\|i-k\|_\infty \ge n$ implies either $\|k\|_\infty \ge n/2$ or $\|i\|_\infty \ge n/2$. Bounding the sum over these two ranges yields
\[
\sum_{\|j\|_\infty\ge n} \big|\E[g(g\circ S^j)]\big| \le \sum_{\|k\|_\infty \ge n/2} \theta_{k,2}(g) \sum_{i\in\Z^d} \theta_{i,2}(g) + \sum_{k\in\Z^d} \theta_{k,2}(g) \sum_{\|i\|_\infty \ge n/2} \theta_{i,2}(g).
\]
This completes the proof.
\end{proof}

For $g\in L^2(\Lambda^{\Z^d})$ with $\E[g]=0$, we can also bound the first moment of the sum of covariances with the one of the physical dependent measures. Using Proposition~\ref{prop:covariance_tail},
\[
\sum_{j\in\Z^d}\|j\|_\infty
\big|\E[g(g\circ S^j)]\big| =
\sum_{n=1}^\infty \sum_{\|j\|_\infty\ge n} 
\big|\E[g(g\circ S^j)]\big|\le
2\Delta_2\sum_{n=1}^\infty \sum_{\|j\|_\infty \ge n/2} \theta_{j,2}(g).
\]
Exchanging the order of summation by Tonelli, we conclude that 
\begin{equation}\label{eq:_first_moment_covariance}
\sum_{j\in\Z^d}\|j\|_\infty
\big|\E[g(g\circ S^j)]\big|\le
2\Delta_2\sum_{j\in\Z^d}\theta_{j,2}(g)\sum_{n=1}^{2\|j\|_\infty}1=4\Delta_2\sum_{j\in\Z^d} \|j\|_\infty\theta_{j,2}(g).
\end{equation}

\begin{prop}\label{prop:decay_covariances_v}
Let $v\in C(X)$ with $\E_\mu[v]=0$, and let $f=v\circ\pi$.
If $\sum_{k=1}^\infty k^d\xi_k^{(2)}<\infty$,  then there is $C>0$ such that
\[
\sum_{\|j\|_\infty \ge n} \big|\E[f(f\circ S^j)]\big| \le \frac{C}{n},
\]
for every $n\ge1$.
\end{prop}

\begin{proof}
Since $\pi$ is measure-preserving, $\E[f]=0$. 
Observe that
\begin{equation}\label{eq:linear_decay}
\sum_{k\ge n}k^{d-1}\xi_k^{(2)}(v)\leq \frac{1}{n}\sum_{k\ge n}k^{d}\xi_k^{(2)}(v)\lesssim n^{-1},
\end{equation}
and so Proposition~\ref{prop:rho_dep_class_q} yields that $\Delta_2<\infty$. 

Applying bound~\eqref{eq:bound_phys_dep_with_xi} and Proposition~\ref{prop:covariance_tail},
$$ \sum_{\|j\|_\infty \ge n} \big|\E[f(f\circ S^j)]\big| \le 2 \Delta_2 \sum_{\|j\|_\infty \ge n/2} \theta_{j,2}(f)\lesssim \sum_{\|j\|_\infty \ge n/2}\xi_{\|j\|_\infty}^{(2)}(v). $$
Reasoning as in Remark~\ref{rem:sum_volume}, there is $C_d>0$ such that 
\[
\sum_{\|j\|_\infty \ge n/2}\xi_{\|j\|_\infty}^{(2)}(v)\le C_d
\sum_{k \ge n/2} k^{d-1} \xi_{k}^{(2)}(v).
\]
The proof is concluded by applying~\eqref{eq:linear_decay}.
\end{proof}

\subsection{Rates in the CLT}

We now state a special case of~\cite[Corollary 2]{Gir18} for a random field generated by a function $g \colon \Lambda^{\mathbb{Z}^d} \to \mathbb{R}$ and evaluated over a Følner sequence $\{\Gamma_n\}_{n\ge1}$. Let $\sigma_g^2$ be defined as in~\eqref{eq:sigma^2_general}. For $\beta>0$ and $q\ge2$, set
\[
C_q(\beta) = \sum_{k=0}^\infty k^{d(1-1/q)+\beta}
\bigg(\sum_{\|i\|_\infty=k}\theta_{i,q}^2(g)\bigg)^{1/2}.
\]

\begin{lem}\label{lem:B-E_Gir18}
For $q>2$, let $g\in L^q(\Lambda^{\Z^d})$ with $\mathbb E[g] = 0$. Assume that there exists $\beta>0$ such that the series $C_q(\beta)$ is convergent.
For $q'=\min\{q,3\}$ and $\gamma>0$, we write
\[
\eta=\max\bigg\{\frac{\gamma(q'-1)d-q'}{2}+1;
-\frac{\gamma\beta q}{2(q+1)}\bigg\}.
\]
If $\sigma_g^2>0$, then there is $\kappa>0$ such that, for any $n\ge1$,
\[
\sup_{t\in\R}\bigg|\P^{\Z^d}\bigg(\frac{\sum_{i\in \Gamma_n} g\circ S^i}{|\Gamma_n|^{1/2}}\le t\bigg)-\Phi(t/\sigma)\bigg|\le \kappa
\bigg(|\Gamma_n|^\eta+\sum_{j\in\Z^d}\big|\E[g(g\circ S^j)]\big|
\bigg|\frac{|\Gamma_n\cap (\Gamma_n-j)|}{|\Gamma_n|}-1\bigg|\bigg).
\]
\end{lem}

The following proof goes by applying~\cite[Corollary 2]{Gir18} where, following the notation of the author, we choose $\alpha=\beta$.

\begin{proof}
Our definitions of the random field $\{g\circ S^j\}_{j\in\Z^d}$, the Følner sequence $\{\Gamma_n\}_{n\ge1}$, and the variance~$\sigma_g^2$ are consistent with the setting of~\cite[Corollary 2]{Gir18}. To apply that corollary, it suffices to verify the convergence of the two series in~\cite[Equation~(16)]{Gir18}. By~\cite[Lemma 1]{Gir18}, this condition is guaranteed by the convergence of the series $C_2(\beta)$ and $C_q(\beta)$, as in our definition. Since $1-1/2\le1-1/q$ and $\theta_{i,2}(g)\le\theta_{i,q}(g)$, it follows that $C_2(\beta)\le C_q(\beta)$. So, we can apply~\cite[Corollary 2]{Gir18}  because $C_q(\beta)$ is convergent by assumption. 

Finally, because under our assumptions the second term of the maximum in~\cite[Corollary 2]{Gir18} dominates the third, the exponent of $|\Gamma_n|$ in our setting reduces exactly to $\eta$.
\end{proof}

For $n\ge1$, we fix henceforth $\Gamma_n=[0,n-1]^d\cap\Z^d$.  It follows that 
\[
|\Gamma_n\cap (\Gamma_n-j)|=\begin{cases}
\prod_{\ell=1}^d(n-|j_\ell|)
&\text{if }\|j\|_\infty\le n-1,\\
0&\text{if }\|j\|_\infty\ge n.
\end{cases}
\]
If $\|j\|_\infty\le n-1$,
\begin{equation*}
\frac{|\Gamma_n\cap (\Gamma_n-j)|}{|\Gamma_n|}=\prod_{\ell=1}^d\bigg(1-\frac{|j_\ell|}{n}\bigg)\ge
\bigg(1-\frac{\|j\|_\infty}{n}\bigg)^d.
\end{equation*}
For $x\ge-1$ and $d\ge1$, Bernoulli inequality gives $1-(1-x)^d\le dx$. So, we get the estimate
\begin{equation}\label{eq:Følner}
1-\frac{|\Gamma_n\cap (\Gamma_n-j)|}{|\Gamma_n|}\le 1-
\bigg(1-\frac{\|j\|_\infty}{n}\bigg)^d\le \frac{d\|j\|_\infty}{n}.
\end{equation}

\begin{prop}\label{prop:correlation_decay}
Let $v\in C(X)$ with $\E_\mu[v]=0$, and define $f=v\circ\pi$. 
Let $p\ge d$ and 
$s>(d+1)(2+d/(p+1))$. 
If $v$ is $p$-regular and $\E[R^s]<\infty$, then there is $C>0$ such that 
\[
\sum_{j\in\Z^d}\big|\E[f(f\circ S^j)]\big|
\bigg|\frac{|\Gamma_n\cap (\Gamma_n-j)|}{|\Gamma_n|}-1\bigg|\le \frac{C}{n},
\]
for all $n\ge1$.
\end{prop}

\begin{proof}
For any $n\ge1$, we write
\[
\sum_{j\in\Z^d}\big|\E[f(f\circ S^j)]\big|
\bigg|\frac{|\Gamma_n\cap (\Gamma_n-j)|}{|\Gamma_n|}-1\bigg| = I_1 +I_2,
\]
where
\[
I_1=\sum_{\|j\|_\infty\le n-1}
\big|\E[f(f\circ S^j)]\big|
\bigg(1-\frac{|\Gamma_n\cap (\Gamma_n-j)|}{|\Gamma_n|}\bigg)
\qquad\text{and}\qquad
I_2= \sum_{\|j\|_\infty\ge n}\big|\E[f(f\circ S^j)]\big|.
\]
Proposition~\ref{prop:regularity_of_rho} with $r=d$ and $q=2$ yields that $\sum_{k=1}^\infty k^d\xi_k^{(2)}(v)<\infty$, where $\{\xi_k^{(2)}(v)\}_{k\ge1}$ is defined in~\eqref{eq:rho_k}. By applying Proposition~\ref{prop:decay_covariances_v}, it follows that $I_2\lesssim n^{-1}$.

By~\eqref{eq:Følner},
\[
I_1\le  \frac{d}{n}\sum_{j\in\Z^d}
\|j\|_\infty \big|\E[f(f\circ S^j)]\big|.
\]
Using the bounds in~\eqref{eq:_first_moment_covariance} and~\eqref{eq:bound_phys_dep_with_xi}, there exists $C>0$ such that  
\[
\sum_{j\in\Z^d}
\|j\|_\infty\, \big|\E[f(f\circ S^j)]\big|
\le  4\Delta_2\sum_{j\in\Z^d}
\|j\|_\infty\, \theta_{j,2}(f)\le 
C \sum_{j\in\Z^d\setminus{\{\bf 0\}}}\|j\|_\infty\, \xi_{\|j\|_\infty}^{(2)}(v)
.
\]
  Reasoning similarly as in Remark~\ref{rem:sum_volume}, there is $C_d>0$ such that
\[
\sum_{j\in\Z^d\setminus{\{\bf 0\}}}\|j\|_\infty\, \xi_{\|j\|_\infty}^{(2)}(v)=
\sum_{k=1}^\infty k|B_k\setminus B_{k-1}|\, \xi_{k}^{(2)}(v) \le C_d\sum_{k=1}^\infty k^d\xi_{k}^{(2)}(v)<\infty.
\]
It follows that $I_1\lesssim n^{-1}$, thus finishing the proof.
\end{proof}

\begin{proof}[Proof of Theorem~\ref{thm:rates_CLT}]
Since $\pi$ is a factor map, it suffices to prove the rate of convergence in the CLT for $f=v\circ\pi$.
Fix $\epsilon\in(0,d/2)$ and $q \ge 3$ as in the statement, and 
let $\gamma=\epsilon/d^2$ and $\beta=(d^2/\epsilon-2d)(q+1)/q$. For $q'=\min\{q,3\}=3$, our choice of $\gamma$ and $\beta$ satisfies
\[
\frac{\gamma(q'-1)d-q'}{2}+1=-\frac{\gamma\beta q}{2(q+1)}.
\]
In addition, $f$ is bounded on $Y$ and hence it belongs to~$L^q$. Using the notation of Lemma~\ref{lem:B-E_Gir18}, we get $\eta=-1/2+\epsilon/d$, and so  $|\Gamma_n|^\eta=(n^d)^\eta=n^{-d/2+\epsilon}$. Assuming that $C_q(\beta)$ is convergent, the proof is finished by applying Lemma~\ref{lem:B-E_Gir18} and Proposition~\ref{prop:correlation_decay}.

We are left to show that $C_q(\beta)$ is convergent.
By \eqref{eq:bound_phys_dep_with_xi}, there is $D>0$ for which
\[\textstyle
\sum_{\|i\|_\infty=k}\theta_{i,q}^2(f)\le D\sum_{\|i\|_\infty=k} \big(\xi_{\|i\|_\infty}^{(q)}(v)\big)^2
\]
for every $k\ge1$. Since $|\{i\in\Z^d:\|i\|_\infty=k\}|\lesssim k^{d-1}$, 
\[\textstyle
\sum_{\|i\|_\infty=k}\theta_{i,q}^2(f)\lesssim k^{d-1}\big(\xi_{k}^{(q)}(v)\big)^2.
\]
Therefore, there exists $D_1>0$ such that 
\begin{align*}
C_q(\beta)&\le \theta_{{\bf0},q}(f)+
D_1\sum_{k=1}^\infty k^{(1-1/q)d+\beta}k^{(d-1)/2}
\xi_{k}^{(q)}(v).
\end{align*}
Note that $(1-1/q)d+\beta+(d-1)/2=h_\epsilon(q)$, where $h_\epsilon(q)$ is defined in~\eqref{eq:h_eps(q)}.
By $\theta_{{\bf0},q}(f)\le 2\|v\|_\infty$,
\[
C_q(\beta)\le 2\|v\|_\infty+D_1\sum_{k=1}^\infty k^{h_\epsilon(q)}
\xi_{k}^{(q)}(v).
\]
By the conditions~\eqref{eq:conditions_p_and_s} on $p$ and $s$, the convergence of $C_q(\beta)$ 
follows by applying Proposition~\ref{prop:regularity_of_rho} with $r=h_\epsilon(q)$, finishing the proof.
\end{proof}

\bibliographystyle{abbrv}
\bibliography{References}
\end{document}